\documentclass[12pt]{article}

\usepackage[margin=1in]{geometry}

\usepackage{diagbox}
\usepackage{mathtools}
\usepackage{bbm}
\usepackage{latexsym}
\usepackage{epsfig}
\usepackage{amsmath,amsthm,amssymb,enumerate}

\usepackage[a-1b]{pdfx}
\usepackage{hyperref}

\usepackage{graphicx}
\newcommand{\Rastar}[0]{~~\mbox{\raisebox{-.05ex}{$\stackrel{*}{\Longrightarrow}$}}~~}

\def\bw{{\bf w}}
\def\bbw{{\bf \bar w}}
\def\mxr{\text{maxr}}

\usepackage{color}
\def\red{\color{red}}

\def\cH{{\mathcal H}}
\def\nn{\nonumber}
\def\a{\alpha}  \def\d{\delta} \def\D{\Delta}
\def\e{\varepsilon}  \def\F{{\Phi}}  \def\g{\gamma}
\def\G{\Gamma}  
\def\z{\zeta} \def\th{\theta}    
   
\def\r{\rho}   
\def\t{\tau} \def\om{\omega}

\def\cP{{\cal P}}
\def\cQ{{\cal Q}}

\newtheorem{theorem}{Theorem}
\newtheorem{lemma}[theorem]{Lemma}

\newcommand{\wh}[1]{\widehat{#1}}

\newcommand{\brac}[1]{\left(#1\right)}

\newcommand{\bfrac}[2]{\left(\frac{#1}{#2}\right)}

\def\cE{{\cal E}}

\newcommand{\set}[1]{\left\{#1\right\}}
\def\sm{\setminus}

\def\es{\emptyset}

\def\E{\mathbb{E}}

\def\Pr{\mathbb{P}}

\def\cF{{\cal F}}
\newcommand{\ignore}[1]{}

\def\cA{{\mathfrak A}}
\def\cB{{\mathfrak B}}
\def\cC{{\mathfrak M}}
\def\cD{{\mathfrak D}}
\def\cE{{\mathfrak N}}
\def\cF{{\mathfrak F}}
\def\cH{{\mathcal H}}

\def\cK{{\mathcal K}}

\def\cM{{\mathcal M}}

\def\cP{{\mathfrak P}}
\def\cQ{{\mathfrak Q}}
\def\cR{{\mathfrak R}}
\def\cS{{\mathcal S}}
\def\cT{{\mathcal T}}

\def\cV{{\mathfrak V}}
\def\cW{{\mathcal W}}
\def\cX{{\mathfrak X}}
\def\cY{{\mathcal Y}}
\def\cZ{{\mathcal Z}}

\newcommand{\card}[1]{\left|#1\right|}

\newcommand{\beq}[2]{\begin{equation}\label{#1}#2\end{equation}}
\newcommand{\mults}[1]{\begin{multline*}#1\end{multline*}}

\def\nn{\nonumber}

\def\cG{\mathcal{F}}

\usepackage{tikz}
\usetikzlibrary{decorations.pathmorphing}
\usetikzlibrary{positioning}
\usetikzlibrary{arrows,automata}
\usetikzlibrary{shapes.misc}
\usetikzlibrary{backgrounds}
\usetikzlibrary{arrows,shapes}

\DeclareMathOperator{\Bin}{Bin}
\DeclareMathOperator{\Hyp}{Hyp}

\begin{document}

\title{The threshold for loose Hamilton cycles in random hypergraphs}
\author{Alan Frieze\thanks{Department of Mathematical Sciences, Carnegie Mellon University, Pittsburgh PA, USA, 15213. Research supported in part by NSF grant DMS1952285.} and Xavier P\'erez-Gim\'enez\thanks{Department of Mathematics, University of Nebraska-Lincoln, Lincoln NE, USA, 68588. Research supported in part by Simons Foundation Grant \#587019 and by NSF grant DMS2201590.}} 
\maketitle
\date{}
\begin{abstract}
We show that w.h.p.\ the random $r$-uniform hypergraph $H_{n,m}$ contains a loose Hamilton cycle, provided $r\geq 3$ and $m\geq \frac{(1+\epsilon)n\log n}{r}$, where $\epsilon$ is an arbitrary positive constant. This is asymptotically best possible, as if  $m\leq \frac{(1-\epsilon)n\log n}{r}$ then w.h.p.\ $H_{n,m}$ contains isolated vertices.
\end{abstract}
\section{Introduction}\label{intro}
The thresholds for the existence of Hamilton cycles in the random graphs $G_{n,m},G_{n,p}$ have been known for many years, see \cite{AKS},  \cite{Boll} and \cite{KS}. There have been many generalisations of these results over the years and the problem is well understood. It is natural to try to extend these results to hypergraphs.

An $r$-uniform hypergraph is a pair $H=(V,E)$ where $E\subseteq \binom{V}{r}$. Let $H_{n,p}$ be the binomial random $r$-uniform hypergraph on $n$ vertices with edge probability $p$. (The uniformity is $r$ thoughout and so we drop $r$ from some of the notation.) Let $H_{n,m}$ be the random $r$-uniform hypergraph on $n$ vertices with $m$ randomly chosen $E_m$ edges from $\binom{V}{r}$. We are also interested in oriented hypergraphs in which each edge has one distinguished vertex called its {\em tail} and $r-1$ {\em heads}. Let $\vec H_{n,p}$ be the oriented counterpart of $H_{n,p}$ in which each one of the $r\binom{n}{r}$ possible oriented edges is picked independently with probability $p$. Note that if we take $\vec H_{n,p}$ and ignore the orientations, then the outcome is distributed as $H_{n,p'}$ with $p'= 1 - (1-p)^r$, so $p'\sim rp$ when $p=o(1)$.

We say that an $r$-uniform sub-hypergraph $C$ of $H$ is a Hamilton cycle of type $\ell$, for some $1\le \ell \le r-1$, if there exists a cyclic ordering of the vertices $V$ such that every edge consists of $r$ consecutive vertices and for every pair of consecutive edges $e_{i-1},e_i$ in $C$ (in the natural ordering of the edges) we have $|e_{i-1}\setminus e_i|=\ell$. When $\ell=r-1$ we say that $C$ is a {\em loose} Hamilton cycle. Frieze \cite{F}, Dudek and Frieze \cite{DF} and Dudek, Frieze, Loh and Speiss \cite{DFLS} established the threshold for the existence of a loose Hamilton cycle in $H_{n,m}$ up to a constant factor. (See also, Ferber \cite{Fe} and Frankston, Kahn, Narayanan and Park \cite{FKNP}.) We will prove the following theorem: 
\begin{theorem}\label{th1}
Let $\e>0$ be an arbitrary small positive constant. If $r\geq 3$ and $(r-1)\mid n$ and $m\geq \frac{(1+\e)n\log n}{r}$ then w.h.p.\ $H_{n,m}$ contains a  loose Hamilton cycle.
\end{theorem}
If $m\leq \frac{(1-\e)n\log n}{r}$ then w.h.p.\ $H_{n,m}$ contains isolated vertices and so this is asymptotically best possible. (The constraint $(r-1)\mid n$ is necessary for the existence of a loose Hamilton cycle. We should also remark that Narayanan and Schacht \cite{NS} obtained the precise threshold for type $\ell$ cycles, except for the case of loose Hamilton cycles.)

We need the following result of Altman, Greenhill,  Isaev and Ramadurai \cite{AGIR}. They proved the following theorem: let $H_{n,d}$ denote a random $d$-regular, $r$-uniform hypergraph on vertex set $[n]$.
\begin{theorem}\label{lcr}
Suppose that $r\mid dn$ and $r-1\mid n$. Then,
\[
\lim_{n\to\infty}\Pr(H_{n,d}\text{ contains a loose Hamilton cycle})=\begin{cases}1&d>\r(r).\\0&d\leq \r(r).\end{cases}
\]
Here $\r=\r(r)$ is the unique real in $(2,\infty)$ such that
\[
(\r-1)(r-1)\bfrac{\r r-d-r}{\r r-\r}^{(r-1)(\r r-\r-r)/r}=1.
\]
\end{theorem}
We also need (the proof of) the following theorem of Kahn \cite{K1}:
\begin{theorem}\label{K1}
Let $\e>0$ be an arbitrary positive constant. If $r\mid n$ and $m\geq \frac{(1+\e)n\log n}{r}$ then w.h.p.\ $H_{n,m}$ contains a  perfect matching.
\end{theorem}
\section{Proof of Theorem \ref{th1}} 
Let $\cK=\binom{[n]}{r}$ denote the edges of the complete $r$-uniform hypergraph with vertex set $[n]$. Let $N=|\cK|=\binom{n}{r}$. Let $\hat H_{n,d}$ be the random (multi)-hypergraph in which every vertex $v$ picks a random set of $d$ edges that contain $v$. Each edge is picked independently of other choices and with replacement, so in principle one vertex could pick the same edge twice or two different vertices may pick the same edge. However, neither of these situations will occur w.h.p~for $d=d^*=\e^2\log n$, which is what we are interested in. We can think of $\hat H_{n,d}$ as a hypergraph generalization of the $d$-out model of random graphs. For convenience, we may sometimes regard $\hat H_{n,d}$ as oriented by making each vertex the tail of all the edges that it picked.

We will be concerned with $H_{n,p}$ for $p=(1+\e)p_*$ where $p_* = \log n / \binom{n-1}{r-1}$ and $\e>0$ is a small constant. 

\begin{theorem}\label{Hout}
Let $d^*=\e^2\log n$, $p=(1+\e) \log n / \binom{n-1}{r-1}$ and $\r\in\mathbb N$ constant. Then, there is a coupling in which $H_{n,p}$ contains $\r$ independent copies ${\wh \G}_1,{\wh \G}_2,\ldots,{\wh \G}_\r$ of $\hat H_{n,d^*}$ w.h.p.
\end{theorem}
\begin{proof}
We will consider three types of edges: $1$, $2$ and $3$. Edges of types $1$ and $3$ are oriented, i.e.~they have a distinguished vertex called the tail and $r-1$ heads. Edges of type $2$ are not oriented. $\vec E_1, E_2, \vec E_3$ will be auxiliary random hypergraphs on $[n]$ with edges of types $1,2,3$, respectively, and built independently of each other. $\vec E_1$ is built by picking each one of the $r\binom{n}{r}$ possible oriented edges with probability $p_1=1 - \left(1-\frac{\e}{2}p_*\right)^{1/r} \sim \frac{\e}{2r}p_*$ independently from other choices. So $\vec E_1$ is distributed as $\vec H_{n,p_1}$.
$E_2$ is distributed like $H_{n,p_2}$ (unoriented) with $p_2=(1+\e/2)p_*$.
Finally, $\vec E_3$ is distributed as $\vec H_{n,p_3}$ with $(1-p_3)^r=1-p_2$, so $p_3\sim p_2/r$. 
Note that if $E_i$ ($i\in\{1,3\}$) is obtained from $\vec E_1,\vec E_3$ by forgetting the orientation of the vertices, then it is distributed as $H_{n,\tilde p_i}$, where $\tilde p_i = 1-(1-p_i)^r \sim rp_i$ (more precisely, $\tilde p_1=\frac{\e}{2}p_*$ and $\tilde p_3 = p_2 =(1+\e/2)p_*$). Therefore, there is no hope of embedding $E_1\cup E_2\cup E_3$ inside of $H_{n,p}$ since each one of $E_2$ and $E_3$ have almost the same number of edges as $H_{n,p}$, so the union contains almost twice the needed edges. So our approach will be to consider all edges of type~1 together with some of type~2 and some of type~3 in a way that we can build the desired coupling.

We generate $\vec E_1$ as the union of $\rho$ independent copies $\vec E^{(1)}_1,\ldots,\vec E^{(\rho)}_1$ of $\vec H_{n,p'_1}$ with $(1-p'_1)^\rho=1-p_1$, so $p'_1\sim p_1/\rho=\frac{\e}{2r\rho}p_*$.
Given a vertex $v$ in an oriented hypergraph, its out-edges are the edges that have $v$ as a tail. Out-degrees and other related notions are defined in the obvious way. We say $v$ is {\em good} if its out-degree is at least $d^*$ in every $\vec E^{(1)}_1,\ldots,\vec E^{(\rho)}_1$. Otherwise we call vertex $v$ {\em bad}. 
\begin{lemma}\label{mind}
Let $H_1,H_2,\ldots,H_\r$ be independent copies of $\vec H_{n,\a p_*}$ for some $\a>0$. Let $V_\a$ denote the set of vertices of degree less than $d^*$ in at least one of the $H_i$. Then $\E(|V_\a|)\leq n^{\th_\a}$ where $\th_\a=\e^2(\log1/\e^2+1+\log\a)- \a$.
 \end{lemma}
 \begin{proof}
 Let $p_\a=\a p_*$. Then, 
\begin{align*}
\\E(|V_\a|)&\leq n\r\sum_{k=0}^{d^*}\binom{\binom{n-1}{r-1}}{k}p_\a^k(1-p_\a)^{\binom{n-1}{r-1}-k}\\
&\leq 2n\r\bfrac{\binom{n-1}{r-1}p_\a e}{\e^2\log n}^{\e^2\log n}\exp\set{-\brac{\binom{n-1}{r-1}-k}p_\a}\\
&\leq 3n\r\bfrac{e\a}{\e^2}^{\e^2\log n}n^{-\a}\\
&=O(n^{1+\th_\a}).
\end{align*}
\end{proof}
Putting $\a=\e/(2r\r)$ in Lemma \ref{mind} and using the Markov inequality, we see that w.h.p.\ all but at most a $1/n^{\e'}$ fraction of the vertices are good, where $\e'=\e'(\e,r,\r)>0$. We call this event $\mathcal F_1$.

Now we will consider a hypergraph that contains all edges of type $1$, those edges of type $2$ with at most one bad vertex and those of type $3$ with more than one bad vertex. More precisely, let $V_G$ be the set of good vertices and $V_B=[n]\setminus V_G$ the set of bad ones. Let $E'_2 = \{e\in E_2: |e\cap V_B|\le 1\}$. We can orient each edge in $E'_2$ with exactly one bad vertex by making that bad vertex the tail, and edges with no bad vertices are oriented following any arbitrary deterministic rule. The resulting oriented hypergraph is denoted $\vec E'_2$.

Let $\mathcal F_2$ be the event that every bad vertex has degree at least $\rho d^*$ in $E'_2$ (or equivalently out-degree at least $\rho d^*$ in $\vec E'_2$). For good vertices we have already found enough edges in $E_1$ so we don't care about their degrees in $E_2$ or $E_3$. Putting $\a=1+\e/2$ in Lemma \ref{mind}, we see that $\th_\a<-1$ and so $V_\a=\es$ w.h.p and $\mathcal F_2$ holds. 

Then define $\vec E'_3 = \{e\in \vec E_3: |e \cap V_B|> 1\}$ and let $E'_3$ be the unoriented version of $\vec E'_3$ resulting from ignoring orientations. Edges in $E'_2$ appear with probability $p_2$ among those edges containing at most one bad vertex. Edges in $E'_3$ appear with probability $1-(1-p_3)^r=p_2$ among the other edges that are not elegible for $E'_2$. In other words, even though each one of $E'_2$ and $E'_3$ depend on the set of good vertices (and thus on $E_1$), $E'_2\cup E'_3$ is distributed as $H_{n,p_2}$ and is independent of $E_1$. Then $E_1\cup E'_2\cup E'_3$ is distributed as $H_{n, \tilde p_1+p_2-\tilde p_1p_2}$ which can be trivially included inside of $H_{n,p}$ since $\tilde p_1+p_2-\tilde p_1p_2\le \tilde p_1+p_2 = p$. Hence, all it remains to show is that in the event that $\mathcal F_1\cap \mathcal F_2$ holds, we can find $\rho$ independent copies of $\hat H_{n,d^*}$ (regarded as oriented) contained in $\vec E_1\cup \vec E'_2\cup \vec E'_3$, which after forgetting orientations we just showed is contained in $H_{n,p}$. We call such copies of $\hat H_{n,d^*}$ $K_1,\ldots,K_\rho$. If $\mathcal F_1\cap \mathcal F_2$ fails, just pick $K_1,\ldots,K_\rho$ independently from everything else so the coupling fails, but this occurs with probability $o(1)$.

Suppose we reveal the out-degrees of $\vec E^{(1)}_1,\ldots,\vec E^{(\rho)}_1$ and degrees of $E'_2$ (but nothing else) and suppose that $\mathcal F_1\cap \mathcal F_2$ holds. In particular, we know which vertices are good or bad.
For every vertex $v\in[n]$ and ``layer'' $i=1,\ldots,\rho$, we want to pick (with replacement) $d^*$ out-edges of $v$ uniformly at random and independently of all other choices. These will be the out-edges of $v$ in $K_i$. If $v$ is good, this is straightforward, since $v$ has out-degree $d_{v,i}\ge d^*$ in $E^{(i)}_1$ and all sets of $d_{v,i}$ out-neighbours are equally likely and independent of those of other vertices and layers. Incorporating replacement is also straightforward. (In other words, we are coupling $d^*$ random independent samples with replacement from a universe with a uniformly chosen random subset of size $d_{v,i}\ge d^*$.)

If $v$ is bad, picking its $d^*$ out-edges in each $K_i$ is trickier. We will be using edges in $\vec E'_2\cup \vec E'_3$ with tail $v$. In particular, we will only use edges in $\vec E'_2$ with exactly one bad vertex, and ignore those with no bad vertices. Recall that edges of type 2 were originally unoriented, but we oriented them in $\vec E'_2$ so that their only bad vertex (if they have one) is the tail. We will also use edges in $\vec E'_3$ that have tail $v$. Most bad vertices will not pick vertices from $\vec E'_3$, but some will. Note that the edge densities of $\vec E'_2$ and $\vec E'_3$ are different since each out-edge of $v$ with only one bad vertex appears in $\vec E'_2$ with probability $p_2$ (before conditioning on degrees of $E_2$) while each out-edge of $v$ with more bad vertices appears in $\vec E'_3$ with probability $p_3\sim p_2/r$. This is a delicate issue that complicates things and is possibly the crux of the argument. We cannot make the density of $\vec E'_2$ smaller a priori since we want to make sure that bad vertices have enough out-edges and we cannot make the density of $\vec E'_3$ larger since otherwise the coupling with $H_{n,p}$ would not work.

Let us proceed to do this carefully. Let $v$ be a bad vertex. Let $X=X_v$ be the set of all possible out-edges of $v$ in $\vec E'_2$, i.e.~edges with tail at $v$ and $r-1$ good vertices as heads. Note $x=|X|=\binom{|V_G|}{r-1}\sim \binom{n-1}{r-1}$. Let $Y=Y_v$
be the set of all possible out-edges of $v$ in $\vec E'_3$, i.e.~edges with tail at $v$ that contain at least one additional bad vertex. Clearly, $y=|Y| = \binom{n-1}{r-1} - x \sim \frac{|V_B|}{n}x \le n^{-\e'}x$ by $\mathcal F_1$.
Since we are also assuming that $\mathcal F_2$ holds, the out-degree of $v$ in $\vec E'_2$ is $s_v\ge \rho d^*$, so the set $S=S_v$ of out-edges of $v$ in $\vec E'_2$ is a random subset of $X$ of size $s_v\ge \rho d^*$. Let $S'$ be a random subset of $S$ of size exactly $\rho d^*$ (which is still uniformly chosen among all subsets of $X$ of that size). Also, let $T=T_v$ be the set of out-edges of $v$ in $\vec E'_3$. Now $t=|T|$ is distributed as $\Bin(y,p_3)$. Our goal is to find a coupling $U\subseteq S'\cup T$ where $U=\{u_1,\ldots,u_{\rho d^*}\}$ is obtained by uniformly and independently sampling $\rho d^*$ random elements of  $X\cup Y$ (where recall $X\cup Y$ is the set of all possible out-edges of $v$) with replacement (i.e.~$|U|\le\rho d^*$ and could be strictly smaller if the same edge is sampled twice). 

To do this, for each $i=1,\ldots,\rho d^*$, let $a_i$ be an integer in $[x+y]= \left[ \binom{n-1}{r-1} \right]$ choosen uniformly at random (and independent of other choices). We can think of the edges in $X\cup Y$ as having distinct labels in $[x+y]$ so we pick $u_i$ by first picking its label $a_i$, and later revealing what edge corresponds to that label. However, all we know so far is that edges in $Y$ have labels in $[y]$ and edges in $X$ have labels in $[x+y]\setminus[y]$. The label of each individual edge is still undecided. The reason to do this is to see: 1) how many different edges are in $U$ after possible edge repetitions, and 2) how many of these fall in $X$ or $Y$.
If $a_i\le y$, that means that $u_i$ will land in $Y$ (which has probability $y/(x+y)\le n^{-\e'}$). Otherwise, it lands in $X$. Let $A=\{a_1,\ldots,a_{\rho d^*}\}$, $A_X=A\setminus [y]$ and $A_Y=A\cap [y]$. Then $u=|U|=|A|$, $u_X=|U\cap X|=|A_X|$ and $u_Y=|U\cap Y|=|A_Y|$. Let $\ell$ be the number of indices $i=1,\ldots,\rho d^*$ such that $a_i\in[y]$ (or equivalently $u_i\in Y$). Clearly, $\ell$ has distribution $\Bin(\rho d^*,y/(x+y))$ and moreover $u_Y\le \ell$ by construction. If $u_Y\le t$, then we can assign the $u_Y$ labels in $A_Y$ to a random set of edges in $T$. Similarly, we can always assign the $u_X\le\rho d^*=|S'|$ labels in $A_X$ to a random set of edges in $S'$. Doing so, the $u_1,\ldots,u_{\rho d^*}$ are i.i.d.~uniformly chosen from $X\cup Y$ as desired. For all of this to work, it is enough to find a coupling in which $\ell\le t$, or in other words $\Bin(\rho d^*,y/(x+y))\le\Bin(y,p_3)$. Ignoring $(1+o(1))$ factors in the parameters, all we need is a coupling
\begin{equation}\label{sillycoupling}
\Bin\left(\rho \e^2 \log n,1/n^{\e'}\right) \le \Bin\left(n^{-\e'}\binom{n-1}{ r-1}, \frac{1+\e/2}{r} \frac{\log n}{\binom{n-1}{ r-1}}\right).
\end{equation}
The coupling in~\eqref{sillycoupling} exists from the lemma below with $N=\rho \e^2 \log n$, $P=1/n^{\e'}$, $K= \frac{r}{1+\e/2} \frac{\binom{n-1}{r-1}}{n^{\e'}\log n}$ and $L = 
\frac{1+\e/2}{\e^2 r\rho}$.

Summarizing, for each bad vertex $v$ we were able to sample $u_1,\ldots,u_{\rho d^*}$ uniformly at random and independently from the set of all $\binom{n-1}{r-1}$ edges that contain $v$ (or oriented edges with tail at $v$) in a way that $u_1,\ldots,u_{\rho d^*}$ are also in $E'_2\cup E'_3$ (assuming $\mathcal F_1$ and $\mathcal F_2$) and thus in $H_{n,p}$. Then for each $i=1,\ldots,\rho$, we just pick the edges from $v$ in $K_i$ to be $u_{1+(i-1)d^*},\ldots,u_{d^*+(i-1)d^*}$. This gives the desired coupling and completes the proof of the theorem.\qed

\begin{lemma}
Let $0\le P\le1/2$ and $N,K,L\in\mathbb N$ with $L\ge2$. Then there is a coupling
\[
\Bin(N,P) \le \Bin(LKN,P/K)
\]
\end{lemma}
\begin{proof}
Let $Y=\sum_{i=1}^N\sum_{j=1}^{LK}Y_{i,j}$ with $Y_{i,j}$ i.i.d.~$\text{Bernoulli}(P/K)$, so $Y$ is $\Bin(LKN,P/K)$. We can write $Y=\sum_{i=1}^N Y'_i$ where $Y'_i = \sum_{j=1}^{LK}Y_{i,j}$. Putting $Z = \sum_{i=1}^N Z_i$ with $Z_i=1_{Y'_i\ge1}$, we have that $Z_i\le Y'_i$ and thus $Z\le Y$. Note that $Z$ is distributed as $\Bin(N,P_0)$ with
\[
1-P_0 = (1-P/K)^{LK} \le e^{-LP} \le e^{-2P} \le 1-P
\]
(where we used that $L\ge 2$ and $0\le P\le1/2$). Since $P\le P_0$, we can easily find a $\Bin(N,P)$ random variable $X$ and a coupling such that $X\le Z$. As a result, $X\le Y$ as desired.
\end{proof}
\section{Shamir problem with a lower bound on minimum degree}\label{Sham}

It will be convenient to replace $\e$ by $2\e$ and work with $H_{n,p},p=(1+2\e)p_*$. This can be coupled to contain the union of $\r=\lfloor \r(r)\rfloor+1$ random hypergraphs $H_0\cup\bigcup_{i=1}^\r H_i$ where $H_0=H_{n,(1+\e)p_*}$ and $H_i=H_{n,\e m_*/\r}$ where $m_*=Np_*$. Next let $\wh\G_i$, $i=1,2,\ldots,\r$, be subhypergraphs of $H_0$ as in Theorem \ref{Hout} and let $\G_i=\wh\G_i\cup H_i,i=1,2,\ldots,\r$.
The aim of this section is to prove the following: 
\begin{theorem}\label{th5}
$\G_{1},\G_2,\ldots,\G_\r$ contain (independent uniform) perfect matchings (assuming that $r$ divides $n$), w.h.p.
\end{theorem}
Theorem \ref{th1} will follow from Theorem \ref{th5} and the following: if $r\geq \lceil e^s\rceil$ then the union of $r$ independent perfect matchings is contiguous to $\cH_{n,,r}$.  This is proved in \cite{Contig}. Symmetry implies that the perfect matchings in the $\G_i$ are uniform and independence follows from the construction and Theorem \ref{Hout}. We apply Theorem \ref{lcr} and the result of \cite{Contig} with the  matchings in $\G_1,\G_2,\ldots,\G_\r$  to obtain Theorem \ref{th1}, under the restriction on $n$ that $r(r-1)\mid n$. This restriction will be removed in Section \ref{div}. 

We continue to the proof of Theorem \ref{th5} and concentrate on $\G_1$. Put quite simply, our proof is a minor  modification of Kahn's proof of Theorem \ref{K1} giving the correct constant in Shamir's problem, see \cite{K1}. We let $e_1,e_2,\ldots,e_N$ be a random permutation of $\cK$. Then for $t\geq 1$, we let $R_t=\set{e_{t+1},e_{t+2},\ldots,e_N}$. We let $\cH_t$ denote the hypergraph with edge-set $E_1\cup R_t$ where $E_1=E(\wh \G_1)$. Thus, we begin with $\cH_0=\cK$ and in a step of the process we obtain $\cH_{t+1}$ from $\cH_t$ by deleting the edge $e_t$ from $R_t$. We leave $E_1$ untouched by this process, so that $m_t=|\cH_t|\sim N-t+nd^*t/N$ w.h.p.\ and $m_t\in [N-t,N-t+nd^*]$. We continue for $T=N-\e n\log n$ steps. Thus the edges of $\G_1=\cH_T$ are $E_1\cup R_T$. We let $E_{1,v}$ denote the $d^*$ choices of edge by vertex $v$ for $E_1$.

In principle, we could just write out Kahn's proof and make a few changes. This is a difficult proof and it would require about another 20 pages. This does not seem efficient and instead, we will give an outline of Kahn's proof and indicate the places where we need to make a change. 

\paragraph{Some notation:} The degree $d_{\cH_t}(v)=d_t(v)$ is the number of edges of $\cH_t$ containing $v$ and the co-degree $d_{\cH_t}(v,w)=d_t(v,w)$ of $v,w$ is the number of edges in $\cH_t$ containing  $v,w$. We let $\cH_t-Z$ denote the subgraph of $\cH_t$ induced by $[n]\sm Z$ where $Z\in \cK$. 
Generally speaking, from now on, we use Fraktur fonts for events, e.g. $\cA,\cB$ etc. and calligraphic fonts for hypergraphs (and sets of edges) e.g. $\cG,\cH$ etc. (By and large we use the notation of \cite{K1}.)

We use the following Chernoff bounds throughout the paper:
\begin{align}
\Pr(|\Bin(n,p)-np|\geq \e np)\leq 2e^{-\e^2np/3}.\label{cher1}\\
\Pr(\Bin(n,p)\geq \a np)\leq \bfrac{e}{\a}^{\a np}.\label{cher2}
\end{align}
These also hold for the hypergeometric distribution $\Hyp(M,K,n)$, where $M$ is the population size, $K$ is the number of successes in the population and $n$ is the number of draws. Then we claim that \eqref{cher1}, \eqref{cher2} hold with $\Bin(n,p)$ replaced by $\Hyp(M,K,n)$ under the assumption that $K/M=p$, see for example Hoeffding \cite{Hoe}, Theorem 4. 

Kahn \cite{K1} defines the following events: for a hypergraph $\cH$, we let $\cM(\cH)$ denote the set of perfect matchings of $\cH$ and let $\F(\cH)=|\cM(\cH)|$. 
 \begin{align*}
 \cA_t=&\set{\log\F(\cH_t)>\log \F(\cH_0)-\sum_{i=1}^t\g_i-o(n)}\quad\text{where $\g_i=\frac{n}{rm_i}$.}\\
 \cR_t=&\{d_t(v)\sim r|\cH_t|/n(\sim D_t=rm_t/n)\text{ for a.e. }v,\,d_t(v,w)=o(D_t)\text{ for all }v,w,\,\D_t=O(D_t),\d_t=\Omega(D_t)\}\\
 &\text{where $\D,\d$ denote maximum and minimum degrees.}\\ \\ 
 \cB_t=& \{\mxr\ \bw_{t}=O(1)\}\\
&\text{where $\bw_{t}(S)=\bw_{\cH_t}(S)=\F(\cH_t-S)$ for $S\subseteq \cK$, $\bbw_t=N^{-1}\sum_{S\in \binom{[n]}{r}}\bw_t(S)$}\\
&\text{and $\mxr\ \bw_t=\frac{\max_{S\subseteq V_t} \bw_t(S)}{\bbw_t}$.}
 \end{align*}
Kahn then writes 
\beq{events}{
\Pr\brac{\bigcup_{t\leq T}\bar \cA_t}\leq \Pr\brac{\bigcup_{t\leq T}\bar \cR_t}+\sum_{t<T}\Pr(\cA_t\cR_t\bar\cB_t)+\sum_{t\leq T}\Pr\brac{\brac{\bigcap_{i<t}\cB_i}\cap \bar\cA_t},
}
where $T=N-\e n\log n$ as defined above. Of course, it suffices to show is that $\cA_T$ occurs w.h.p. 

\subsection{Property $\cR_t$}
Fix a vertex $v$. Suppose first that $t\leq \binom{n}{2}-n^{3/2}$. In this case $D_t=rm_t/n\gg d^*$ \footnote{We write $A_n\gg B_n$ to mean $A_n/B_n\to\infty$ and $A_n\sim B_n$ if $A_n=(1+o(1))B_n$ as $n\to\infty$.} A vertex in $\cH_t$ has a degree in $R_t$ that is hypergeometrically distributed with mean $rm_t/n\geq rn^{1/2}$. The Chernoff bound \eqref{cher1} implies that w.h.p.\ every $v\in[n]$ has degree asymptotic to $rm_t/n$. Thus $d_t(v)\sim D_t$ for all $v\in [n]$ and also that $\D_t=O(D_t),\d_t=\Omega(D_t)$. The co-degree of $v,w$ is bounded by $d^*=o(D_t)$ plus the co-degree of $v,w$ in $R_t$. But the latter is distributed hypergeometrically with mean $O(m_tn^{-2})$. The Chernoff bounds imply that this is $o(D_t)$ q.s.\footnote{A sequence of events $\cE_n,n\geq 0$ occurs quite surely (q.s.) if $\Pr(\cE_n)\geq 1-O(n^{-K})$ for any constant $K$.} and this verifies $\cR_t$ for $t\leq N-2n^{3/2}$, say.

Now assume that $m_t< n^{3/2}$. Then $\E(|E_1\cap R_t|)=O(n\log n\times m_tn^{-r})=O(m_tn^{1-r}\log n)=o(1)$ and so $E_1\cap R_t=\es$ w.h.p.\ for $t>N+n^{3/2}$. ($E_1$ is unchanged and $R_t$ decreases as the construction proceeds.) The argument for degrees now follows that in the previous paragraph given that the degree of a vertex $v$ is its degree in $H_1$ plus its degree in $R_t$. The Chernoff bounds imply that $d_t(v)\in (1\pm \log^{-1/3}n)r(m_t/n+d^*)$ for all but $o(n)$ vertices. Thus, this time we only claim that $d_t(v)\sim D_t$ for almost all $v\in [n]$. We now check co-degrees. Fix $v,w$. Then, let $Z_1$ denote the number of edges in $E_{1}$ coming from $v,w$'s choices and let $Z_2$ be the number of edges containing $v,w$ due to choices of $x\neq v,w$. Then $Z_1$ is is bounded in distribution by the sum of two copies of $\Bin(d^*,\frac{r-1}{n-1})$ and $Z_2$ is bounded in distribution by $\Bin((n-2)d^*,\frac{(r-1)(r-2)}{(n-1)(n-2)})$. So
\[
\Pr(Z_1\geq 4)\leq \bfrac{d^*r}{n}^4\text{ and }\Pr(Z_2\geq 4)\leq \bfrac{r^2nd^*}{n^2}^4.
\]
Now let $\wh d_t(v,w)$ denote the co-degree of $v,w$ in $R_t$. Then, where $p_t=(N-t+1)/N$,
\beq{cov}{
\Pr(\exists v,w\ s.t.\ \wh d_t(v,w)\geq k)\leq \binom{n}{2}
\frac{\binom{\binom{n-2}{r-2}}{k}\binom{N-k}{N-t+1}}{\binom{N}{N-t+1}}\leq \binom{n}{2}\binom{\binom{n-2}{r-2}}{k}p_t^k\leq n^2\bfrac{r^2m_t}{kn^2}^k.
}
We put $k=3r$ and see that $d_t(v,w)=o(D_t)$ w.h.p., for all $v,w,t$. Thus  $\Pr\brac{\bigcup_{t\leq T}\bar \cR_t}=o(1)$.
 
 We finish this section with the observation that 
 \beq{H-Z}{
 \text{if $Z\in \cK$ and $\cH_t\in \cR_t$ then $\cH_t-Z\in \cR_t$.}
 }
To see this note that if $v\notin Z$ then its degree in $\cH_t-Z$ is at most $r\max_{v,w} d_t(v,w)$ less than it is in $\cH_t$.\\
(To avoid problems with applying \eqref{H-Z} repeatedly, we only apply this when $|\cH_t|\to\infty$.)
\subsection{Property $\cA_t$}
In this section we discuss the proof that
\[
\sum_{t\in I}\Pr\brac{\brac{\bigcap_{i<t}\cB_i}\cap \bar\cA_t}=o(1).
\]
We know from \cite{K1} that w.h.p. the R.H.S. of \eqref{events} is $o(1)$ provided that $T\leq T_1=N-Kn\log n$ for a sufficiently large $K$. So, w will assume that $t\geq T_1$ from now on. Observe next that $\E(|E_1\cap R_{T_1}|)=O((n\log n)^2/n^3)=o(1)$ and so we can assume that $E_1\cap T_t=\es$ for $t\geq T_1$.

We now claim that 
\beq{onestep}{
\E(\F(\cH_{t+1})\mid \cH_{t})=\F(\cH_{t})(1-\xi_t),
}
where 
\[
\E(\xi_t)\sim \g_{t}=\frac{n}{rm_t}.
\]

For this we will need to define the following event: $\cE_\t=\set{\bw_\cH(A)\sim \F(\cH)/D_\t\text{ for a.e. }A\in \cH_\t}$ where $D_\t=rm_\t/n=1/\g_\t$ is the average vertex degree of $\cH_\t$. So, assume that $\cE_{t}$ occurs. For $M\in\cM(\cH_{t})$ we have
\[
\Pr(M\notin \cM(\cH_{t+1}))=\frac{n}{r}\cdot \frac{|M\cap R_{t}|}{|R_{t}|}.
\]
On the other hand, the occurence of $\cE_t$ and $\cB_t$  implies that
\beq{eM}{
\frac{n}{r}\sum_{M\in \cM(\cH_t)}|M\cap R_{t}|\sim \frac{m_t\F(\cH_t)}{D_t}.
}
This is because the LHS of \eqref{eM} counts the number of pairs $(e,M)$ where $e\in M\in \cM(\cH_t)$ and the RHS does this approximately, assuming $\cE_t$ and $\cB_t$.

It follows that 
\[
\E(\xi_t\mid \cH_t)=\frac{1}{\F(\cH_t)}\sum_{M\in \cM(\cH_t)}\frac{n}{r}\cdot \frac{|M\cap R_{t}|}{|R_{t}|m_t}\sim \frac{1}{D_t}=\g_t,
\]
as required.

After this there is a careful use of martingale tail inequalities whose validity relies on the occurence of the $\cB_i$ to show that we always have $\xi_i=O(\g_i)$. This gives us a concentration of $\sum_i\xi_i$ around its mean and also $\sum_i\xi_i^2=o(n)$. The argument in \cite{K1} remains valid in our case.

Thus w.h.p.
\[
\log \brac{\F(\cH_0)\prod_{i=1}^t(1-\g_i)}\geq \log \F(\cH_0)-\sum_{i=1}^t\g_i-o(n).
\]
\subsection{Property $\cB_t$}
The aim here is to prove that
\beq{main}{
\sum_{t\leq T}\Pr(\cA_t\cR_t\bar\cB_t)=o(1).
}
For this, Kahn \cite{K1} introduced several new events: in the following we have dropped the subscript $t$.  $D=D_{\cH}=r|\cH|/n$ is the average degree in $\cH=\cH_t$; $Z$ ranges over $\cK$ and $A$ ranges over the edges of $\cH$.
\begin{align*}
\cP_x&=\set{\bw_{\cH}((Z\sm x)\cup y)\gtrsim \bw_{\cH}(Z)d_\cH(x)/D\text{ for a.e. }y\in [n]\sm Z}.\\
\cC&=\set{x\in Z\text{ and }\bw_{\cH}(Z)\geq \F(\cH)e^{-o(n)}\text{ implies }\cP_x}.\\
\cE&=\set{\bw_\cH(A)\sim \F(\cH)/D\text{ for a.e. }A\in \cH}.\\
\cF&=\set{\bw_\cH(Z)\sim \F(\cH)/D\text{ for a.e. }Z\in \cK}.
\end{align*}
Kahn proves four lemmas: for events $\cE_1,\cE_2$, the notation $\cE_1\Rastar\cE_2$ means that\\ $\Pr(\cE_1\cap \neg\cE_2)=n^{-\omega(1)}$. 
\begin{enumerate}[{\bf KLemm{a} 6.1}]
\item If $\cH$ satisfies $\cA,\cR$ then it satisfies $\cE$.
\item $\{\cH\text{ satisfies } \cA,\cR\} \Rastar \{\cH\text{ satisfies } \cF\}$.
\item For $x\in Z\in \binom{[n]}{r}$, \{$\cH\text{ satisfies } \cR\}\wedge \{\cH-Z\text{ satisfies } \cF\}\Rastar \{(\cH,Z,x)\text{ satisfies } \cP_x$\}.
\item If $\cH$ satisfies $\cR,\cF,\cC$ then it satisfies $\cB$.
\end{enumerate}
The proof of \eqref{main} follows easily from these lemmas. Given $\cA,\cR$, KLemma 6.1 implies that we have $\cE$. Then $\cH-Z$ satisfies $\cR$ (see \eqref{H-Z}) and if $\bw_{\cH}(Z)\geq \F(\cH) e^{-o(n)}$ then $\cH-Z$ satisfies $\cA$. Thus, given KLemma 6.2 we have that q.s., $\cH-Z$ satisfies $\cF$ whenever $\bw_\cH(Z)\geq \F(\cH) e^{-o(n)}$. Then, KLemma~6.3 gives us that $\cC$ holds q.s. KLemma 6.4 then implies that $\cB$ is satisfied and \eqref{main} follows. Property $\cP_x$ is key here. Starting with $Z$ maximizing $\bw_t$, it shows that $\mxr\bw_t=O( \bbw_t)$.

Now $\cH$ has a different distribution to that considered in \cite{K1}. But by looking at the proofs in \cite{K1} we shall see that this does not matter. Klemmas 6.1 and 6.4 are actually deterministic statements. It is only the proofs of KLemmas 6.2,6.3 that are sensitive to the distribution of $\cH$. Fortunately, our $\cH$ meets the necessary requirements. We briefly indicate the main points of the proofs of these lemmas to show why the change in distribution does not matter.
\paragraph{Klemma 6.1}
The proof in \cite{K1} begins with "Here $\cH$ is a general $m$-edge ($n$-vertex) $r$-graph satisfying $\cA,\cR$." Given $\cH$ satisfying $\cA,\cR$, {\red Kahn first writes the lower bound on $\log\F(\cH)$ can be expressed as 
\[
\tfrac{n}{r}\log D-\Lambda -o(n)>\frac{1}{r}\sum_{v}\log d(v)-\Lambda -o(n). 
\]
where $d(v)$ is the degree of $v$ in $\cH$ and we have used Jensen's inequality.

It is important here to use $m/\binom{n}{r}$ instead of $p$ in the derivation. Kahn then shows that $h(v,\cH)>\log d(v)-o(1)$ for a.e. vertex $v$ where $h(v,\cH)$ is the entropy of the distribution of the edge containing $v$ in a uniform random perfect matching of $\cH$.} This shows that in a random p.m. $v$ is contained in a nearly random incident edge and $\cE$ quickly follows. 
\paragraph{Klemma 6.2}
This uses the technical assumption that  
\beq{tech}{
|\cK\sm\cH|\geq cN
}
 for some absolute constant $c>0$. Kahn points out that without this assumption $\cE$ and $\cF$ are equivalent.

Given $\cH$ satisfying $\cA,\cR$, let $\th$ be an arbitrary positive constant, let $\F'=\F(\cH)/D$ and let
\mults{
\cQ=\text{\{$\bw_{\cH}(A)\sim \Phi'$ for a.e. $A\in  \cH$,}\\
\text{ but $\bw_{\cH}(U)\neq (1\pm 2\theta) \Phi'$ for
at least a $(2\theta)$-fraction of the $U$'s in $ \cK\sm\cH$}\}.
}
Kahn reduces the lemma to showing that 
\beq{K100}{
\Pr(\cA\cR\cQ)<n^{-\om(1)}.
}
Let $\cT$ be chosen uniformly from $\binom{\cH}{\t}$ for a suitable $\log n\ll \t\ll \log^2n$ and let $\cG=\cH\sm\cT$. Let $\z=e^{-\t/D}$ and define the event 
\[
\cV=\{\bw_{\mathcal F}(A)\sim \z \F'\text{ for a.e. }A\in \cT
\text{ and at least a $\theta$-fraction of the $U$'s in $ \cK\sm\cH$ lie in $\cW$}\},
\]
where $\cW=\set{U:\bw_{\mathcal F}(U)\neq (1\pm\th)\z\F'}$.

Note that \eqref{K100} is implied by
\begin{align}
\Pr(\cV\mid \cA\cR\cQ)&=1-o(1).\label{85}\\
\Pr(\cV)&=n^{-\om(1)}.\label{86}
\end{align}
For equation \eqref{85} Kahn carefully chooses $\t$ and then observes "For the proof of \eqref{85} we choose $\cH$ and then $\cT$. We assume we have chosen $\cH$ satisfying $\cA,\cR,\cQ$; so $\Pr$ now refers just to the choice of $\cT$, and \eqref{85} will follow from $\Pr(\cH, \cT\mbox{ satisfies }\cV )=1- o(1)$".

For equation \eqref{86} we choose $\cG=\cH\sm\cT$ and then $\cT$. More precisely, we need to choose a disjoint pair $\cG,\cT$ such that $\cG\cup\cT$ is distributed as $\cH$ and $\cT$ is distributed as a uniform random subset of $\cG\cup\cT$. So we choose $\cH'$ with the same distribution as $\cH$, then choose a random subgraph $\cT'$ of $\cH'$ of size $\t$.  The hypergraph $\cH'\sm\cT'$ will be our $\cG$. Then we choose $\cT=\cT_1\cup\cT_2$ independently as follows: we choose $\cT_1$ of size $\Bin(\t,\a_1),\a_1=(N-t+1)/(N-t+1+d^*n)$ from $\cK\sm \cG$ and $\cT_2$ of size $\t-|\cT_1|$ from $E_1\sm\cG$. Note that $\a_1\geq \e/(\e+\e^2)\geq 0.99$ and so $|\cT_1|\geq 9\t/10$ q.s..

Given $\mathcal F$, let $U_1,U_2,\ldots,$ be an ordering of $\cK\sm\mathcal F$ with $\bw_{\mathcal F}(U_1)\leq \bw_{\mathcal F}(U_2)\leq \cdots$ and let $\cY,\cZ$ be the first and last $\th|\cK\sm\mathcal F|/3$ of the $U_i$'s. Now at least one of $\cY,\cZ$ is contained in $\cW$. Otherwise, using $\t\ll |\cK\sm\cH|$ (from \eqref{tech}), we have $|\cW\sm\cH|<|\cY|+|\cZ|<2\th(|\cK\sm\cH|+\t)/3<\th|\cK\sm\cH|$, contradicting the second part of $\cV$. 

 Now the first part of $\cV$ implies that $\min\set{|\cT\cap\cY|,|\cT\cap\cZ|}\leq \th\t/10$, say.  But $|\cT_1\cap \cY|$ is distributed as $\Bin(|\cT_1|,\th/3)$ and because $|\cT_1|\geq 9\t/10$ q.s. we have that  $\min\set{|\cT_1\cap \cY|,|\cT_1\cap \cZ|}\geq \th\t/4$ q.s. This contradiction implies that \eqref{86} holds.
\paragraph{Klemma 6.3}
Let $\cH'=\{A\in \cH:A\cap Z=\{x\}\}$ and $\cH''=\cH\sm \cH'$. Let $Y=Z\sm x$ and $W=[n]\sm Z$. Property $\cR$ for $\cH$ implies that
\begin{align}
D_{\mathcal G}&\sim D_\cH, \quad\text{where }\mathcal G=\cH-Z.\nn\\
d'(x)&= |\cH'| ~\brac{=\card{\set{S\in \binom{W}{r-1}: S\cup x\in \cH}}}  \sim d_{\cH}(x)=\Omega(D).\label{72}
\end{align}
We choose $\cH''$ first, which determines $\mathcal G$. We can assume \eqref{72} and 
\beq{73}{
\bw_{\mathcal G} (U )\sim \F'=\frac{\F(\mathcal G)}D\text{ for a.e. }U\in \binom{[n]\sm Z}{r}.
}
After this we choose $\cH'$ and KLemma 6.3 will follow from
\beq{74}{
\Pr(\cP_x\mid \eqref{72},\eqref{73})=1-n^{-\om(1)}.
}
From \eqref{73} we have that for a.e. $y\in W$, 
\beq{75}{
\bw_{\mathcal G}(S\cup y)\sim \F'\text{ for a.e. }S\in \binom{W\sm y}{r-1}.
}
For a $y$ as in \eqref{75} and a random choice of $\cH'$ (as described in Section \ref{intro}), the Chernoff bounds imply that q.s., for all but $o(d'(x))$ sets $S$ such that $S\cup x\in \cH$, we have $\bw_{\mathcal G}(Y\cup y)\sim \F'$. As this is another place where the distribution of $\cH$ matters, we go into a little more detail. So, given $y$,  let $\cS$ denote the collection of sets $S\in \binom{W\sm y}{r-1}$ such that \eqref{75} does not hold and let $\eta=|\cS|/\binom{n-r-1}{r-1}=o(1)$. Then the number $Z_1$ of edges in $R_t$ of the form $S\cup x\in \cH',S\in \cS$ is  dominated in distribution by $\Bin(N-t+1,r\eta/n)$ and so $Z_1=o(D)$ q.s. The same argument shows that the number $Z_2$ of edges in $E_{1,v}$ of the form $S\cup x\in \cH',S\in \cS$ satisfies $Z_3=o(D)$ q.s. Because $d'(x)=\Omega(D)$, we have $Z_1+Z_2=o(d'(x))$ as required.

Now
\beq{70}{
\text{if $y\in W$ then }\bw_{\cH}(Y\cup y)=\sum_{S\in \binom{W\sm y}{r-1},S\cup x\in\cH}\bw_{\mathcal G}(S\cup y).
}
And so $\bw_\cH(Y\cup y)\gtrsim \F'd'(x)\sim \bw_\cH(Z)d_\cH(x)/D$.
\paragraph{Klemma 6.4}
We define the property
\[
\cD=\{\text{If $\bw_\cH(Z_0)>\F e^{-o(n)}$ then $\bw_\cH(Z)\gtrsim \bw_\cH(Z_0)D^{-r}\prod_{z\in Z_0}d_\cH(x)$ for a.e. $Z\in\cK$\}.}
\]
After this we assert:
\begin{align}
\cR\cC&\text{ implies }\cD.\label{67}\\
\cR\cD\cF&\text{ implies }\cB.\label{68}
\end{align}
If $Z_0=\{x_1,\ldots,x_r\}$ then we obtain \eqref{67} by induction on $i\in[r]$, for a.e. $y_1,\ldots,y_r\in [n]\sm Z_0$, we have with $Z_i=(Z_{i-1}\sm x_i)\cup y_i$,
\beq{69}{
\forall i\ \bw_\cH(Z_i)\gtrsim \bw_\cH(Z_{i-1})d_\cH(x_i)/D\gtrsim \bw_\cH(Z_0)D^{-i}\prod_{j\leq i}d_\cH(x_j).
}
(The only thing to observe here is that \eqref{69} for $Z_{i-1}$ implies $\bw_{\cH}(Z_i)\geq \F(\cH)e^{-o(n)}$, since $\d_\cH=\Omega(D)$ implies that the r.h.s.\ of
\eqref{69} is $\Omega(\bw_\cH(Z_0))$.) This gives $\cD$, since it implies that a.e.\ $Z\in \cK$ is $Z_r$ for some $y_1\dots y_r$ supporting \eqref{69}.


For \eqref{68} choose $Z_0\in \cK$ with $\bw_\cH(Z_0)$ maximum and note that $\bw_\cH(Z_0)>\F(\cH)e^{-o(n)}$. Thus $\cD$ (and $\d_\cH=\Omega(D)$) give $\bw_\cH(Z)=\Omega(\bw_\cH(Z_0))$ for a.e.\ $Z\in \cK$, which with $\cF$ implies $\Phi(\cH)/D=\Omega(\bw_\cH(Z_0))$.  But this gives $\cB$, since $\overline \bw_\cH(\cH)=\Phi(\cH)/D$. 
\section{Dealing with the divisibility problem}\label{div}
We use an idea of Ferber \cite{Fe}. Suppose that $n=(r-1)(r\hat n+k)$ where $0\leq k<r$. We can use $k=0$ as the base case for an inductive proof. So, suppose that $\cX_k$ is the statement that $H_{n,m}$ contains a loose hamilton cycle w.h.p.\ when $(r-1)\mid n$ and $n/(r-1) \equiv k \mod r$ and $m\geq \frac{(1+\e)n\log n}{r}$ for any fixed $\e>0$. Assume then that $\cX_k$ is true. We will show using Ferber's idea that if $n=(r-1)(r\hat n+k+1)$ then there exists a positive constant $\a$ such that if $m\geq \frac{(1+\e)n\log n}{r}$ then $H_{n,m}$ contains a loose Hamilton cycle with probability at least $\alpha$. (In fact, we will show that $\alpha=1/(2r)$ works.) After this, we can use a result of Friedgut \cite{F05} that says that there are no coarse thresholds with $m\sim cn\log n$ for a constant $c>0$. In particular, we apply the second remark following Theorem 2.1 of that paper to obtain w.h.p. 

Following Ferber, we let $p=m/N$ and then let $p_1=\log n/N$ and define $p_2$ by $1-p=(1-p_1)(1-p_2)$. Fix an arbitrary edge $e=\set{x_1,x_2,\ldots,x_k}\in E(H(n,p_1))$ and note that such an edge exists w.h.p. Now let $V^*=([n]\sm e)\cup \set{e^*}$, where $e^*$ is a new vertex. Note that $|V^*|= n - r+1 =(r-1)(r\hat n+k)$. Now consider the hypergraph $H^*$ with vertex set $V^*$ derived from $H_{n,p_2}$ by replacing an $x_i$ in any edge by $e^*$ and accepting it with probability $\frac{1-(1-p_2)^{1/r}}{p_2} \sim  1/r$. The edge density of $H^*$ is such that we can apply the induction hypothesis to argue that $H^*$ contains a loose Hamilton cycle $C^*$ w.h.p. Now by symmetry, there is a probability of at least $1/r$ that $\set{e^*}=e_1^*\cap e_2^*$ for two consecutive edges of $C^*$. Suppose that $e_i^*$, $i=1,2$, are derived from $e_1,e_2$ of $H_{n,p_2}$. Given this there is a further probability of at least $1-1/r\geq 1/2$ that we can replace $e_1^*,e_2^*$ by $e_1,e,e_2$ to obtain a loose Hamilton cycle in $H_{n,p_2}$.
\end{proof}
\section{Final remarks}
Kahn \cite{K2} also proved a hitting time version for perfect matchings. It is not obvious how to tighten the current approach to get a hitting time version for loose hamilton cycles. This is because we need a minimum degree of $\r>2$ to be able to apply Theorem \ref{lcr}.

\end{document}